\documentclass[12pt,leqno]{article}

\usepackage{amsmath,amssymb,amscd,amsthm}
\usepackage{latexsym,enumerate,comment}
\usepackage{graphicx}
\usepackage{xcolor}
\usepackage{mathrsfs}
\usepackage{mathabx}
\usepackage{authblk}

\usepackage{float}

\theoremstyle{definition}
\newtheorem{thm}{Theorem}[section]
\newtheorem{lm}[thm]{Lemma}
\newtheorem{cor}[thm]{Corollary}
\newtheorem{ex}[thm]{Example}
\newtheorem{rem}[thm]{Remark}
\newtheorem{as}[thm]{Assumption}

\newtheorem{defn}[thm]{Definition}
\newtheorem{prop}[thm]{Proposition}

\newcommand\vp{\varphi}

\renewenvironment{proof}{\vskip 10pt{\noindent\bf Proof.\ }}{\hfill$\blacksquare$\vskip 10pt}

\numberwithin{equation}{section}

\begin{document}

\title{Total Positivity and Spectral Properties of Linearized Operators}
\author{John Albert and Steve Levandosky}
\date{}

\maketitle

\begin{abstract}
For a class of semilinear elliptic equations, we establish criteria that guarantee that the linearized operator associated with a solution satisfies certain spectral assumptions that are widely used in the analysis of the stability of solitary waves. The criteria only involve the symbol of the linear operator and positivity and symmetry of the solution, and can therefore be verified without an explicit formula for the solution. 
\end{abstract}

\section{Introduction}

In this paper we consider linearized operators which arise in stability analysis of traveling-wave solutions of nonlinear wave equations, and study spectral properties of these operators which play an important role in the stability theory.  We will consider here equations of the form
	\begin{equation}\label{E:evolution}
	u_t-(\mathcal{M}u)_x+(f(u))_x=0\qquad x\in\mathbb R, t\geq0
	\end{equation}
where $f: \mathbb R \to \mathbb R$ and $\mathcal{M}$ is a linear operator. Equations of this form arise as models for waves in dispersive media.  For example, the equation corresponding to the choice $\mathcal{M}=-\partial_x^2$ is the well-known KdV equation, while  $\mathcal{M}=\partial_x^4-b\partial_x^2$ corresponds to the Kawahara equation \cite{B:kawahara}.  In such settings one is often interested in the stability of traveling waves. We define a traveling-wave solution of \eqref{E:evolution} to be a solution of the form $u(x,t)=\vp(x-ct)$, where $\vp \in L^2(\mathbb R)$ is called the wave profile. Thus $\vp$ is a traveling-wave profile if and only if it is a solution in $L^2(\mathbb R)$ of the equation 
\begin{equation}
\label{E:traveling}
	\mathcal{L}\vp=f(\vp),
\end{equation}
where $\mathcal{L}=\mathcal{M}+c$.

Existence of solutions of \eqref{E:traveling} has been proven under a variety of assumptions on $\mathcal{L}$ and $f$. For the purposes of this paper we restrict attention to operators $\mathcal{L}$ defined by 
\begin{equation}
\mathcal{L}v=\mathcal{F}^{-1}\left[\alpha(\xi)\cdot\mathcal{F}v(\xi)\right],
\label{defL}
\end{equation}
where we make the following assumption on $\alpha(\xi)$.

\begin{as}\label{A:alpha_assumption} 
The function $\alpha(\xi)$ is measurable on $\mathbb{R}$, and there exist $c_1 >0$, $c_2>0$ and $s>\frac12$ such that 
	\begin{equation}
    \label{E:L}
	    \qquad c_1(1+|\xi|^{2s})\leq \alpha(\xi)\leq c_2(1+|\xi|^{2s}) \quad \text{for all $\xi$ in $\mathbb R$}.
	\end{equation}
\end{as}

This assumption guarantees in particular that $\mathcal{L}$ is a bounded operator from $H^q(\mathbb{R})$ to $H^{q-2s}(\mathbb{R})$, for every $q \in (-\infty,\infty)$.
 
%This is equivalent to the assumption that $m(\xi)$ is locally integrable and bounded below, with $c > -\inf_{\xi \in \mathbb R} m(\xi)$, and satisfies $m_1 |\xi|^{2s}\le |m(\xi)| \le m_2|\xi|^{2s}$ for all sufficiently large $\xi \in \mathbb R$, with $m_1, m_2 > 0$.

Concerning the nonlinear term in \eqref{E:evolution}, we make the following assumption throughout.

\begin{as}\label{A:f_assumption}  The function $f(r)$ is twice continuously differentiable on $\mathbb{R}$, and satisfies 

\begin{enumerate}
\item[(F1)] $f(0)=0$ and $f(r)>0$ for all $r>0$,
\item[(F2)] there exists $p>1$ such that $r^2f'(r)\geq prf(r)$ for all $r\in\mathbb R$,
\item[(F3)] there exist $2\leq q_1\leq q_2$ and $C>0$ such that $|f'(r)|\leq C(|r|^{q_1-1}+|r|^{q_2-1})$ for all $r\in\mathbb R$.
\end{enumerate}

\end{as}

Solutions of \eqref{E:traveling} in $H^s(\mathbb R)$ are critical points of the action functional 
	\[
	S(u):=\int_{\mathbb R} {\textstyle\frac12}u\mathcal{L}u-F(u)\,dx
	\]
where $F(u):=\int_0^u f(t)\,dt$.   One possible strategy for proving existence of such critical points is to use the method of concentration compactness to show that $S$, when constrained to the Nehari manifold 
\[
	\mathcal{N}:=\{u\in X: u\neq0, \left<S'(u),u\right>=0\},
	\]
has a nontrivial minimizer.  This was accomplished in \cite{B:esfahani_levandosky2021} in  the case where $\mathcal{L}$ is a fourth-order differential operator, and the method of proof used there generalizes to arbitrary $\mathcal{L}$ and $f$ satisfying the assumptions above, yielding the following result.

\begin{thm}\label{T:existence} Let $\mathcal{L}$ have symbol $\alpha$ satisfying Assumption \ref{A:alpha_assumption} for some $s>\frac12$, and let $f$ satisfy Assumption \ref{A:f_assumption}. Then there exists a nontrivial solution $\vp\in H^s(\mathbb R)$  of \eqref{E:traveling}.
\end{thm}

We say a traveling-wave profile $\vp$ is {\it stable} in a subspace $X$ of $L^2(\mathbb R)$ if $\vp \in X$ and for any $\epsilon>0$ there exists $\delta>0$ such that for any $u_0\in X$ satisfying $\|u_0-\vp\|_X<\delta$, the solution $u(t)$ of \eqref{E:evolution} satisfies $\inf\{\|u(t)-\vp(\cdot-r)\|_X:r\in\mathbb R\}<\epsilon$ for all $t>0$. We say $\vp$ is {\it unstable} in $X$ otherwise. There is an extensive literature on the  stability of traveling waves of \eqref{E:evolution} and related equations (see for example \cite{B:angulo, B:geyer_pelinovsky, B:kapitula_promislow, B:kowalcyk_martel_munoz} and the references therein), much of which is devoted to proving general conditions implying stability or instability.

Central to the stability analysis of traveling waves is the linearized operator $\mathcal{H}$ associated with a solution $\vp$, defined by
	\begin{equation}\label{E:H_definition}
	\mathcal{H}v:=\mathcal{L}v-f'(\vp)v.
	\end{equation}
One easily sees by differentiating \eqref{E:traveling} with respect to $x$ that $\vp'\in\ker\mathcal{H}$. It is also typically the case that $\mathcal{H}$ has at least one negative eigenvalue. (For the class of equations considered here this follows from \eqref{E:Hphiphi}). A common assumption in the literature is that these eigenspaces are minimal. That is, it is assumed that $\mathcal{H}$ satisfies the following hypotheses.
 
\begin{enumerate}
\item[(H1)] $\mathcal{H}$ has exactly one negative eigenvalue, which is simple.
\item[(H2)] $\ker\mathcal{H}=\mathrm{span}\{\vp'\}$.
\item[(H3)] the rest of the spectrum of $\mathcal{H}$ is positive and bounded below.
\end{enumerate}
 
For example, these hypotheses appear prominently in the abstract framework of Grillakis, Shatah and Strauss \cite{B:grillakis_shatah_strauss}. Using a modification of this framework, Bona, Souganidis and Strauss \cite{B:bona_souganidis_strauss} proved that, for a class of operators $\mathcal{M}$ and nonlinear terms $f$, if $\vp_c$ is a family of traveling-wave solutions of \eqref{E:evolution} depending smoothly on the wavespeed $c$, then  $\vp_c$ is nonlinearly stable when $\frac{d}{dc}\|\vp_c\|_{L^2}^2 >0$ and unstable when $\frac{d}{dc}\|\vp_c\|_{L^2}^2 <0$, provided hypotheses (H1), (H2), and (H3) hold. %Additional stability criteria were obtained in \cite{B:albert}, \cite{B:albert_bona} and \cite{B:albert_bona_henry} under the same assumptions on $\mathcal{H}$.
These hypotheses also appear frequently in the linear stability analysis of solitary waves. See, for example, \cite{B:kapitula} and \cite{B:stanislavova_stefanov}.

\begin{comment}
The following result is well-known. See, for example, \cite{B:albert,B:albert_bona,B:bona_souganidis_strauss,B:grillakis_shatah_strauss}.
\begin{thm}\label{T:stability} Assume \eqref{E:evolution} is locally well-posed in $X$ and that for each $c>0$ there exists a solution $\vp_c \in X$ of \eqref{E:traveling}. Further suppose that, as an operator on $L^2(\mathbb R)$, $\mathcal{H}$ satisfies the following conditions for some $c_0>0$.
\begin{enumerate}
\item[(H1)] $\mathcal{H}$ has exactly one negative eigenvalue, which is simple.
\item[(H2)] $\ker\mathcal{H}=\mathrm{span}\{\frac{d}{dx}\vp_{c_0}\}$.
%\item[(H3)] the rest of the spectrum of $\mathcal{L}$ is positive and bounded below.
\end{enumerate}
Then $\vp_{c_0}$ is stable in $X$ if $\frac{d}{dc}\|\vp_c\|_{L^2}^2>0$ at $c=c_0$, and  $\vp_{c_0}$ is unstable in $X$ if $\frac{d}{dc}\|\vp_c\|_{L^2}^2<0$ at $c=c_0$.
\end{thm}
\end{comment}

We next note that, for the class of operators $\mathcal{L}$ considered here, standard arguments from spectral theory show that we need only concern ourselves with the first two of (H1), (H2), and (H3). 

\begin{thm}\label{T:abh} Suppose Assumptions \ref{A:alpha_assumption} and \ref{A:f_assumption} hold, and let $\varphi$ be as in Theorem \ref{T:existence}.  If the operator $\mathcal{H}$ satisfies hypotheses  (H1) and (H2), then it also satisfies (H3).
\end{thm}
 
\begin{proof} As shown in the proof of Proposition 1 of \cite{B:albert_bona_henry}, under the given assumptions the spectrum of $\mathcal{H}$ contains the interval $[c,\infty)$, and outside this interval the spectrum consists only of a set of isolated eigenvalues, and is bounded below.  It follows that the spectrum of $\mathcal{H}$ consists of the interval $[c,\infty)$ together with either a finite or countable number of eigenvalues, whose only possible accumulation point is at $c$. (Note that here we have corrected an inaccuracy in the statement of Proposition 1 in \cite{B:albert_bona_henry}.) The statement of the theorem follows immediately.
\end{proof}

Verification of (H1) and (H2) is often a nontrivial matter.  In his seminal paper \cite{B:benjamin} on the stability of solitary waves, Benjamin established that (H1) and (H2) are satisfied when  $\mathcal{L}=-\partial_x^2+c$, $f(u)=u^2$, and $\varphi$ is a KdV soliton profile.  Later authors studying the nonlinear Schr\"odinger equation (see, e.g., \cite{B:chang, B:weinstein}) have shown that these hypotheses also hold in all dimensions $d \ge 1$, for nonlinearities $f(u)=u^p$, $1< p < \max\{1+4/(n-2), \infty\}$, when $\mathcal{L}=-\Delta+c$ and $\varphi$ is the unique positive radial solution of \eqref{E:traveling},   provided (H2) is replaced by the condition that $\ker\mathcal{H}$ is spanned by $\{\vp_{x_i}\}_{i=1}^d$.  For nonlocal operators $\mathcal{L}$, hypotheses (H1) and (H2) have been verified in some important special cases as well \cite{B:albert_bona, B:bennett_brown_et_al, B:frank_lenzmann, B:frank_lenzmann_silvestre}.  As yet, a general theory is lacking, and the topic remains an active area of research. 

Our main result is related to that which appears in \cite{B:albert,B:albert_bona}, where it was shown that (H1) and (H2) hold when the Fourier transforms of $\vp$ and $f'(\vp)$ satisfy certain positivity conditions. Before giving the precise statement of this result we recall the following terminology.  See \cite{B:karlin}, a standard reference for the theory of totally positive functions and Polya frequency functions.

\begin{defn} 
\label{totalpositivity}
Let $G(x,y)$ be a function defined on $\mathbb R^2$.  Let $\Delta = \{(x_1, x_2) \in \mathbb R^2: x_1 < x_2\}$, and define the function $G_2(\bar x,\bar y)$  on $\Delta\times\Delta$  by
\begin{equation}
	G_2(\bar{x},\bar{y}):=G(x_1,y_1)G(x_2,y_2)-G(x_1,y_2)G(x_2,y_1),
	\label{subscript2function}
\end{equation}
where $\bar x =(x_1,x_2)$ and $\bar y=(y_1, y_2)$.
\begin{enumerate}[(a)]
\item We say that $G$ is {\it totally positive of order 2} ($TP(2)$)   if $G(x,y) \ge 0$ for all $(x,y) \in \mathbb R^2$ and  $G_2(\bar x, \bar y) \ge 0$ for all $(\bar x, \bar y) \in \Delta \times \Delta$. 
\item We say that $G$ is {\it strongly $TP(2)$} if $G$ is $TP(2)$ and, in addition, satisfies $G(x,x)>0$ for all $x \in \mathbb R$ and  $
G_2(\bar x,\bar x) > 0$ for all $\bar x \in \Delta$.
\item We say that $G$ is {\it strictly $TP(2)$} if $G(x,y)>0$ for all  $(x,y) \in \mathbb R^2$ and  $G_2(\bar x, \bar y) > 0$ for all $(\bar x, \bar y) \in \Delta \times \Delta$. 
\item If $K(x)$ is a function defined on $\mathbb R$, and $G(x,y) = K(x-y)$, we say that $K$ is a {\it Polya frequency function of order 2} ($PF(2)$) if $G$ is $TP(2)$, and we say $K$ is {\it strongly (resp.\ strictly) $PF(2)$} if $G$ is strongly (resp.\ strictly) $TP(2)$.
\end{enumerate}
\end{defn}

\begin{rem} Clearly, strictly $PF(2)$ implies strongly $PF(2)$, which in turn implies $PF(2)$.  
\end{rem}

The arguments used to prove Theorems 3 and 4 from \cite{B:albert_bona} and Theorem 3.2 from \cite{B:albert} are easily generalized to obtain the following result.

\begin{thm} \label{T:albert_bona}  Suppose $\alpha$ satisfies Assumption \ref{A:alpha_assumption} and $f$ satisfies Assumption \ref{A:f_assumption}.   Let $\mathcal{L}$ be given by \eqref{defL}, and let $\vp$ be a solution of \eqref{E:traveling}.  If $\alpha$ is even, $\vp$ and $\mathcal{F}\vp$ are positive and even, and $\mathcal{F}(f'(\vp))$ is strictly $PF(2)$,  then $\mathcal{H}$ satisfies the conditions (H1) and (H2). 
\label{thmAB}
\end{thm}

%\begin{thm}\label{T:albert_bona}  (\cite{B:albert,B:albert_bona}) Let $\vp$ be a solution of \eqref{E:traveling}.  Suppose $\alpha$ is positive and even, $\mathcal{F}\vp$ is positive and even, and $\mathcal{F}(f'(\vp))$ is strictly $PF(2)$.  Then $\mathcal{H}$ satisfies Assumption \ref{A:H_assumption}. 
%\label{thmAB}
%\end{thm}

\begin{rem}
\label{R:oldresult}
Using the arguments in the proof of Theorem \ref{T:positivity} below, one can show that it is sufficient to assume that  $\mathcal{F}(f'(\vp))$ is strongly $PF(2)$.
\end{rem}

In cases where an explicit formula for $\vp$ is known, one can often verify directly that $\mathcal F(f'(\vp))$ is in $PF(2)$.  In fact, since it is known that positive functions whose logarithm is concave on $\mathbb R$ are in $PF(2)$ (cf.\ Lemma \ref{L:PF2log} below), verification of this condition amounts to a simple check that the second derivative of $\log \mathcal F(f'(\vp))$ is positive.  See, for example, \cite{B:kabakouala_molinet,B:natali}, where this approach was applied to a family of explicit solutions of the Kawahara equation. There are, however, many situations where one can prove existence of a solution $\vp$ of \eqref{E:traveling}, but an explicit formula is not known \cite{B:esfahani_levandosky2021}. Thus it is desirable to have sufficient conditions for (H1) and (H2) that do not rely on a formula for $\vp$. The purpose of this paper is to present a new set of such conditions, and to illustrate its application to some equations for which no explicit formula for $\varphi$ is available. Our main result is the following.

\begin{thm}\label{T:positivity}    Suppose $\alpha$ satisfies Assumption \ref{A:alpha_assumption} and $f$ satisfies Assumption \ref{A:f_assumption}.  Let $\mathcal{L}$ be given by \eqref{defL}, and let $\vp$ be a solution of \eqref{E:traveling}. If $\alpha$ is even, $\vp$ and $f'(\vp)$ are positive and even, and the kernel $K$ defined by 
\begin{equation}\label{E:K_definition}
    K(x):=\mathcal{F}^{-1}\left(1/\alpha(\xi)\right)
\end{equation}
is strongly $PF(2)$, then $\mathcal{H}$ satisfies the hypotheses (H1) and (H2). 
\end{thm}

%\begin{thm}\label{T:positivity} Let $\vp$ be a solution %of \eqref{E:traveling}. Suppose $\vp$ and $f'(\vp)$ are positive and even, and the kernel $K$ defined by 
%\begin{equation}\label{E:K_definition}
 %   K(x):=\mathcal{F}^{-1}\left(1/\alpha(\xi)\right)
%\end{equation}
%is even and  strongly $PF(2)$. Then $\mathcal{H}$ %\end{thm}

 Theorems \ref{T:albert_bona} and \ref{T:positivity} are in some sense dual to one another. Whereas Theorem \ref{T:albert_bona} assumes that $\mathcal F(f'(\vp))$ is strictly $PF(2)$, Theorem \ref{T:positivity} assumes instead that the kernel $K(x)$ in \eqref{E:K_definition} is strongly $PF(2)$. 
 
The advantage of having the $PF(2)$ requirement be on the kernel $K$ instead of $\mathcal F(f'(\vp))$ is that, for a number of model equations of the form \eqref{E:traveling} which arise in applications, one can easily verify directly that $K$ is strongly $PF(2)$ (see Lemma \ref{L:convolution_exponential_PF2}), although the solution $\vp$ may not be known explicitly.
  Note that it can often be shown that a positive even solution $\varphi$ of \eqref{E:traveling} exists even when an explicit formula for $\varphi$ is not available; as is illustrated for example in Theorem \ref{T:existence_positive_even} below.

 The plan of the paper is as follows.  We prove our main result, Theorem \ref{T:positivity}, in Section \ref{sec:main}.  In Section \ref{sec:class_of_examples} we identify a class of examples to which the main theorem applies, and in Section \ref{sec:numerical} we describe its application to two specific equations in detail, including some numerical results. 

 \vskip 10pt

We conclude this introduction with a couple of comments on our notation. 
 Our definition of the Fourier transform of a function $f$ on the line is   
    \[\mathcal Ff (\xi) = \int_{\mathbb R}e^{ix\xi} f(x)\ dx.
    \]
With this definition, the inverse Fourier transform of $f(k)$ is given by 
    \[\mathcal F^{-1} f(x) = \frac{1}{2\pi}\int_{\mathbb R} e^{-ix\xi}f(\xi)\ d\xi,
    \]
and for $h(\xi)=\mathcal F f(\xi) \mathcal\cdot F g(\xi)$ we have
 $ \mathcal F^{-1} h(x) = (f \ast g) (x)= \int_{\mathbb R} f(x-y)g(y)\ dy$, when the integral is defined.

 For $s \in \mathbb{R}$ we use $H^s(\mathbb{R}^n)$ to denote the Sobolev space of tempered distributions $f$ on $\mathbb{R}^n$  whose Fourier transforms are locally integrable functions satisfying $(1+|\xi|^2)^{s/2}\mathcal{F}f(\xi) \in L^2(\mathbb{R}^n)$.

\vskip 10pt

\section{Proof of the main result} 
\label{sec:main}

In this section we prove Theorem \ref{T:positivity}. We first note that when the operator $\mathcal{L}$ satisfies \eqref{E:L}, then it is invertible as a map from $X=H^s(\mathbb R)$ to its dual $X^*=H^{-s}(\mathbb R)$. Denote by $\left<w,v\right>$ the pairing of $w\in X^*$ with $v\in X$, and note that for $w\in L^2$ this coincides with the $L^2$ inner product $(w,v)_{L^2}$. We may define an inner product on $X$ by
	\[
	(u,v)_X:=\left<\mathcal{L}u,v\right>
	\]
and the associated norm by $\|u\|_X:=\sqrt{(u,u)_X}$. We note that since 
\[\left<S''(u)v,w\right>=\left<\mathcal{L}v-f'(u)v,w\right>\] for any $u,v,w\in X$, $S''(\vp)$ is equivalent to the linearized operator $\mathcal{H}$ defined by \eqref{E:H_definition}. Since $\vp\in L^\infty(\mathbb R)$, $\mathcal{H}$ maps $X$ to $X^*$. It will be useful to consider the related operator $\mathcal{H}_0:X\to X$ defined by
	\begin{equation}\label{H0_definition}
	\mathcal{H}_0v:=\mathcal{L}^{-1}\mathcal{H}v=v-\mathcal{L}^{-1}(f'(\vp)v).
	\end{equation}

We first show that it suffices to prove $\mathcal{H}_0$ satisfies hypotheses (H1) and (H2).

\begin{thm}\label{T:H0H} If  $\mathcal{H}_0$ satisfies hypotheses (H1) and (H2), then so does $\mathcal{H}$.
\end{thm}

\begin{proof} Since $\mathcal{L}$ is invertible, $\ker\mathcal{H}=\ker\mathcal{H}_0$, so $\mathcal{H}$ satisfies (H2). We next show that $\mathcal{H}$ has at most one linearly independent eigenfunction with a negative eigenvalue. Suppose to the contrary that $\mathcal{H}$ has linearly independent eigenfunctions $\psi_1$ and $\psi_2$ in $X$ with eigenvalues $\lambda_1\leq\lambda_2<0$. We may assume $\psi_1$ and $\psi_2$ are orthogonal in $L^2$ and have unit norm in $L^2$. Let $V=\mathrm{span}\{\psi_1,\psi_2\}$ 
%and let $C=\max\{\|\psi_1\|_X,\|\psi_2\|_X\}$. 
Given $v\in V$ we can write $v=d_1\psi_1+d_2\psi_2$ to obtain 
%	\[
%	\|v\|_X\leq C(|d_1|+|d_2|)\leq\sqrt{2}C\|v\|_{L^2},
%	\]
%so
	\begin{align*}
	(\mathcal{H}_0v,v)_X&=\left<\mathcal{H}v,v\right>
	=(-\lambda_1d_1\psi_1-\lambda_2d_2\psi_2,d_1\psi_1+d_2\psi_2)_{L^2}\\
	&=-\lambda_1d_1^2-\lambda_2d_2^2\\
	&\leq -\lambda_1\|v\|^2_{L^2},
	\end{align*}
so $(\mathcal{H}_0v,v)_X<0$ for all nonzero $v$ in $V$.  Now denote by $\lambda<0$ the unique negative eigenvalue of $\mathcal{H}_0$, $\psi$ a corresponding eigenfunction, and $V^+:=\mathrm{span}\{\psi,\vp'\}^\perp$. Since $0$ and $\lambda$ are the only eigenvalues of $\mathcal{H}_0$ and the remainder of the spectrum of $\mathcal{H}_0$ is bounded below by some positive constant $c_0$, we have $(\mathcal{H}_0p,p)_X\geq c_0\|p\|_X^2$ for any $p\in V^+$. We may now write
	\begin{align*}
	\psi_1&=a_1\psi+b_1\vp'+p_1\\	
	\psi_2&=a_2\psi+b_2\vp'+p_2
	\end{align*}
where $p_1,p_2\in V^+$.  Next note that $a_1=0$ would imply
	\[
	(\mathcal{H}_0\psi_1,\psi_1)_X=(\mathcal{H}_0p_1,b_1\vp'+p_1)_X=(\mathcal{H}_0p_1,p_1)_X\geq0
	\]
which contradicts the fact that $(\mathcal{H}_0\psi_1,\psi_1)_X=\left<\mathcal{H}\psi_1,\psi_1\right>=-\lambda_1\|\psi_1\|_{L^2}^2<0$. Hence $a_1\neq0$ and likewise $a_2\neq0$. Since $\psi_1$ and $\psi_2$ are linearly independent, $v:=a_2\psi_1-a_1\psi_2\in V$ is therefore nonzero so $(\mathcal{H}_0v,v)_X<0$. On the other hand, $v=(a_2b_1-a_1b_2)\vp'+a_2p_1-a_1p_2$, and $p:=a_2p_1-a_1p_2\in V^+$, so 
	$
	(\mathcal{H}_0v,v)_X=(\mathcal{H}_0p,p)_X\geq0$. This contradiction proves $\mathcal{H}$ has at most one linearly independent eigenfunction with a negative eigenvalue. It must have at least one negative eigenvalue since its essential spectrum is bounded below by a positive constant and by (F2)
	\begin{equation}\label{E:Hphiphi}
	\left<\mathcal{H}\vp,\vp\right>_{L^2}=\int_{\mathbb R}\vp f(\vp)-\vp^2f'(\vp)\,dx\leq-(p-1)\int_{\mathbb R}\vp f(\vp)\,dx
	=-(p-1)\int_{\mathbb R}\vp\mathcal{L}\vp\,dx<0.
	\end{equation}
Thus $\mathcal{H}$ satisfies (H1). 
% By Theorem \ref{T:abh}, $\mathcal{H}$ satisfies (H3).
\end{proof}

The proof of Theorem \ref{T:positivity} relies on the following result concerning integral operators with nonnegative kernels.

\begin{prop}\label{P:spectral_prop} Let $\Omega$ be an open, connected domain in $\mathbb R^n$ and let $\mu$ be a positive measure on $\Omega$. Suppose $G\in L^2(\Omega\times \Omega,\mu\times\mu)$ is continuous and satisfies
\begin{enumerate}[(i)]
    \item $G(x,y)=G(y,x)$ for all $(x,y)\in \Omega\times \Omega$.
    \item $G(x,y)\geq0$ for all $(x,y)\in \Omega\times \Omega$.
    \item $G(x,x)>0$ for all $x\in \Omega$
\end{enumerate}
For $w\in L^2(\Omega)$ define
    \begin{equation}\label{E:T_definition}
    T(w)(x):=\int_\Omega G(x,y)w(y)\,d\mu(y).
    \end{equation}
Then $T$ is a compact, self-adjoint operator on $L^2(\Omega,\mu)$, $T$ has a simple eigenvalue $\lambda_0>0$ and all other eigenvalues of $T$ satisfy $|\lambda_j|<\lambda_0$. Every nonzero eigenfunction for $\lambda_0$ must be either positive for all $x \in \Omega$ or negative for all $x \in \Omega$.
\end{prop}

\begin{proof} Since $G\in L^2(\Omega\times\Omega)$, $T$ is a Hilbert-Schmidt operator on $Y:=L^2(\Omega,\mu)$ and is therefore compact. The symmetry of $G$ implies $T$ is self-adjoint. Therefore there exists an orthonormal basis of $L^2(\Omega,\mu)$ consisting of eigenfunctions $\{\psi_i\}_{i=0}^\infty$ of $T$, whose corresponding eigenvalues are real, have finite multiplicity, and can accumulate only at $0$. 

Since $G$ is $L^2$ and continuous, it follows that $T(w)$ is continuous for any $w\in Y$, and thus eigenfunctions with nonzero eigenvalues are continuous.  

From the spectral theorem for compact self-adjoint operators on a Hilbert space, it follows that
\begin{equation}
\lambda_0 = \sup_{\psi \in Y,\ \|\psi\|_Y = 1} ( T\psi, \psi)_Y.
\label{maxeig}
\end{equation}
and that every function $\psi \in Y$ which satisfies $\|\psi\|_Y = 1$ and $(T\psi,\psi)_Y=\lambda_0$ must be an 
eigenfunction of $T$ for the eigenvalue $\lambda_0$.  But if $\psi \in Y$ and $\|\psi\|_Y = 1$, then
 $w = |\psi|$ satisfies $\|w\|_Y=1$, and since $G(x,y)\ge 0$ for all $x,y \in\Omega$, we have that 
 $(Tw,w) \ge (T\psi,\psi)$.  Hence if $\psi$ is an eigenfunction of $T$ for $\lambda_0$, then so is $w = |\psi|$. 

We claim that  if $\psi_0$ is an nonzero eigenfunction for $\lambda_0$, then $|\psi_0(x)| > 0$ for all $x \in \Omega$.  To prove this, let $\Omega_0$ be the set of all $x \in \Omega$ such that $\psi_0(x)=0$ and suppose for contradiction that $\Omega_0$ is nonempty. Then for any $x_*\in\Omega_0$ we have
    \[
    0=\lambda_0|\psi_0(x_*)|=T(|\psi_0|)(x_*)
    =\int_\Omega G(x_*,y)|\psi_0(y)|\,d\mu(y).
    \]
By the continuity of $G$ and the fact that $G(x_*,x_*)>0$, there exists $\delta>0$ such that $B_\delta(x_*)\subseteq\Omega$ and $G(x_*,y)>0$ for all $y\in B_\delta(x_*)$. Since $G(x_*,y)|\psi_0(y)|$ is nonnegative for all $y\in\Omega$ it then follows that $\psi_0(y)=0$ for all $y\in B_\delta(x_*)$. This proves $\Omega_0$ is open.  But $\Omega_0$ is clearly also closed, so since $\Omega$ is connected, we have $\Omega_0=\Omega$, implying  that $\psi_0$ is the zero function in contradiction to our assumption.  Therefore $\Omega_0$ must be empty, which establishes the claim.

From the claim it follows that if $\psi_0$ is an nonzero eigenfunction for $\lambda_0$, then $\psi_0(x) \ne 0$ for all $x \in \Omega$, and from the continuity of $\psi_0$ it follows that either $\psi_0(x)>0$ for all $x \in \Omega$ or $\psi_0(x)<0$ for all $x \in \Omega$.
Since no two functions with this property can be orthogonal in $Y$, it follows that $\lambda_0$ must be a simple eigenvalue of $T$ in $Y$.

We now claim that $T$ cannot have any negative eigenvalues $\sigma$ such that $|\sigma| \ge \lambda_0$.  Suppose to the contrary that $\sigma$ is such an eigenvalue, with corresponding eigenfunction $\psi_\sigma$ satisfying $\|\psi_\sigma\|_Y = 1$.
Then
 $$
\begin{aligned}
(T|\psi_\sigma|,|\psi_\sigma|)_Y &=
\int_{\Omega}\int_{\Omega} G(x,y)|\psi_\sigma(x)| |\psi_\sigma(y)|\ d\mu(x)\ d\mu(y) \\
&\ge \left|\int_{\Omega}\int_{\Omega} G(x,y)\psi_\sigma(x) \psi_\sigma(y)\ d\mu(x)\ d\mu(y) \right| \\
&=|(T\psi_\sigma,\psi_\sigma)_Y| = |\sigma| \ge \lambda_0, 
\end{aligned}
$$
from which it follows that $|\psi_\sigma|$ is an eigenfunction for $\lambda_0$.  Therefore, as shown above, we must have $|\psi_\sigma(x)|>0$ for all $x \in \Omega$.  Hence $\psi_\sigma(x) \ne 0$ for all $x \in \Omega$, and from the continuity of $\psi_\sigma$ it follows that $\psi_\sigma$ is either everywhere positive or everywhere negative on $\Omega$.  But this is impossible, since $\psi_\sigma$ must be orthogonal to $\psi_0$.  This contradiction proves the claim.
\end{proof}

\noindent{\bf Proof of Theorem \ref{T:positivity}.}
By Theorem \ref{T:H0H} it suffices to show $\mathcal{H}_0$ satisfies (H1) and (H2). The proof follows that of Theorem 4 in \cite{B:albert_bona}. 
Define
	\[
	T_1 v:=\mathcal{L}^{-1}(f'(\vp)v)
	\]
for $v\in X$. Since $T_1 v=\lambda v$ if and only if $\mathcal{H}_0v=(1-\lambda)v$, it suffices to prove that $1$ is a simple eigenvalue of $T_1$ and that $T_1$ has a unique simple eigenvalue $\lambda_0>1$. Note that $T_1 v=K*(f'(\vp)v)$, so $T_1$ is an operator of the form \eqref{E:T_definition} with $\Omega=\mathbb R$, $d\mu_1=f'(\vp)\,dx$ and $G_1(x,y)=K(y-x)$. By the growth assumption \eqref{E:L} on $\alpha$, we have $1/\alpha\in L^2(\mathbb R)$, and therefore $K\in L^2(\mathbb R)$ and thus $K^2\in L^1(\mathbb R)$. Since $\vp\in H^s(\mathbb R)$ with $s>\frac12$, we have $\vp\in L^\infty(\mathbb R)$ and by (F3) it follows that $f'(\vp)\in L^2(\mathbb R)$. By Young's inequality we then have $K^2*f'(\vp)\in L^2(\mathbb R)$, and therefore
    \begin{equation}\label{E:K_bound}
    \iint_{\mathbb R^2}K(x-y)^2f'(\vp(x))f'(\vp(y))\,dy\,dx=\left<K^2*f'(\vp),f'(\vp)\right>_{L^2(\mathbb R)}<\infty
    \end{equation}
which proves $G_1\in L^2(\mathbb R^2,\mu_1\times\mu_1)$. Since $K$ is continuous, positive and even, it follows that $G_1$ is continuous and satisfies conditions (i), (ii) and (iii) in Proposition \ref{P:spectral_prop}. 

Next define an operator $T_2$ on $Y_2:=L^2(\Delta,\mu_2)$ by
    \[
	T_2 w(\bar{x}):=\iint_\Delta G_2(\bar{x},\bar{y}) w(\bar{y})\ d\mu_2(\bar{y}),
	\]
where $\mu_2=\mu_1\times\mu_1$ and 
    \[
G_2(\bar{x},\bar{y}):=G_1(x_1,y_1)G_1(x_2,y_2)-G_1(x_1,y_2)G_1(x_2,y_1)
    \]
for $\bar x = (x_1,x_2) \in \Delta$ and $\bar y = (y_1,y_2) \in \Delta$.  It follows from \eqref{E:K_bound} that $G_2\in L^2(\Delta\times\Delta,\mu_2\times\mu_2)$. The assumption that $K$ is strongly $PF(2)$ implies $G_1$ is strongly $TP(2)$, so $G_2$ satisfies (i), (ii) and (iii) in Proposition \ref{P:spectral_prop}.

Therefore, by Proposition \ref{P:spectral_prop}, there exist orthonormal bases of $Y_1$ and $Y_2$ consisting of eigenfunctions of $T_1$ and $T_2$, respectively. Denote by $\{\psi_j,\lambda_j\}_{j=0}^\infty$ the eigenfunctions and corresponding eigenvalues of $T_1$ and by $\{\zeta_j,\kappa_j\}_{j=0}^\infty$ those of $T_2$, where $\lambda_0$ and $\kappa_0$ are simple eigenvalues such that $\lambda_0>|\lambda_j|$ and $\kappa_0 > |\kappa_j|$ for $j\geq 1$. The operators $T_1$ and $T_2$ are related by the identity
    \begin{equation}
	T_2(v\wedge w)=T_1(v)\wedge T_1(w)
    \label{E:T2_of_wedge}
\end{equation}
for any $v,w\in Y_1$, where
	\[
	(v\wedge w)(x_1,x_2):=v(x_1)w(x_2)-v(x_2)w(x_1).
	\]
It then follows that if $\psi_i\wedge\psi_j$ is nonzero, it is an eigenfunction of $T_2$ with eigenvalue $\lambda_i\lambda_j$. Exactly as in the proof of Lemma 9 of \cite{B:albert_bona}, we see that $\lambda_0 \lambda_1 = \kappa_0$. 

We claim now that if $\psi_1$ is any nonzero eigenfunction of $T_1$ for the eigenvalue $\lambda_1$, then $\psi_1$ must be odd, and must vanish only at $x=0$.  This may be proved using the same arguments as in the proof of Theorem 4 in \cite{B:albert_bona}, but for the reader's convenience we give the details here. First, note that since $\alpha$ and $\vp$ are even, the operator $T_1$ preserves parity, and hence the even part $\psi_1^e$ of $\psi_1$, defined by $\psi_1^{e}(x):=\frac12(\psi_1(x)+\psi_1(-x))$, satisfies $T_1(\psi_1^e)= \lambda_1 \psi_1^e$.  Therefore \eqref{E:T2_of_wedge} implies that $T_2(\psi_1^e \wedge \psi_0)=\kappa_0 (\psi_1^e \wedge \psi_0)$, and from Proposition \ref{P:spectral_prop} it then follows that $\psi_1^e \wedge \psi_0$ is either identically zero on $\Delta$ or does not vanish at all on $\Delta$.  Now since $\psi_1^e$ and $\psi_0$ belong to eigenspaces of $T_1$ for distinct eigenvalues, they must be orthogonal in $L^2(\mathbb{R},\mu_1)$, and since $\psi_1^e$ is even and $\psi_0$ is of one sign on $\mathbb{R}$, it follows that $\psi_1^e$ must have at least two distinct zeros $x_1$ and $x_2$ in $\mathbb{R}$, with say $x_1 <  x_2$.  But then $\psi_1^e \wedge \psi_0$ vanishes at $(x_1, x_2) \in \Delta$, and hence must vanish everywhere in $\Delta$. Now for all $x \in \mathbb{R}$ we have either $x <  x_2$, in which case $(\psi_1^e \wedge \psi_0)(x, x_2) = 0$ implies $\psi_1^e(x)=0$, or $x > x_1$, in which case $(\psi_1^e \wedge \psi_0)( x_1, x) = 0$ again implies $\psi_1^e(x)=0$.  This shows that $\psi_1^e$ is identically zero on $\mathbb R$, and hence $\psi_1$ is odd.  

If $\psi_1(x)$ were to have two distinct zeros in $\mathbb{R}$, then since $T_2(\psi_1 \wedge \psi_0)=\kappa_0(\psi_1 \wedge \psi_0)$, the same argument as used above for $\psi_1^e \wedge \psi_0$ would show that  $\psi_1 \wedge \psi_0$ is identically zero on $\Delta$. But then for all $x>0$ one would have $(\psi_1 \wedge \psi_0)(0,x)=0$, and so $\psi_1(x)=0$; while for all $x<0$ one would have $(\psi_1\wedge \psi_0)(0,x)=0$, and so again $\psi_1(x)=0$.  Thus it would follow that $\psi_1$ is identically zero on $\mathbb{R}$, a contradiction.  This proves $\psi_1$ can only vanish at $x=0$.

Since two odd functions which each vanish only at zero cannot be orthogonal in $L^2(\mathbb{R},\mu_1)$, it follows from the claim we have just proved that $\lambda_1$ must be a simple eigenvalue of $T_1$. 

Finally, note that $\vp'$ is an eigenfunction of $T_1$ with eigenvalue 1. Since $\vp$ is positive and even, then $\vp'$ is odd and vanishes only at $x=0$. Thus $(\psi_1,\vp')_{Y_1}\neq0$, so $\psi_1$ and $\vp'$ cannot be eigenfunctions with distinct eigenvalues. This proves that $\lambda_1=1$. Since we already know that $\lambda_0$ is simple and is the only eigenvalue of $T_1$ greater than $\lambda_1$, the proof of the theorem is complete.
\hfill$\blacksquare$

%\begin{lm}
%The eigenvalues $\lambda_0$ and $\lambda_1$ of $T$ are positive and 
%simple, and the eigenvalue $\mu_0$ of $T_2$ is positive and simple.  
%Moreover, we have $\mu_0 = \lambda_0 \lambda_1$.  If $\psi$ is an 
%eigenfunction of $T$, then, up to multiplication by a real constant, we 
%have $\psi(x)>0$ for all $x \in \mathbb R$.  Every eigenfunction of $T$ 
%for the eigenvalue $\lambda_0$ must be everywhere positive and everywhere 
%negative on $\mathbb R$, as must every eigenfunction of $T_2$ for the 
%eigenvalue $\mu_0$. (Also, I think every eigenfunction of $T$ for the 
%eigenvalue $\lambda_1$ must be odd and can vanish only at $x=0$.)
%\label{eigenvaluelemma}
%\end{lm}

\vskip 10pt

 \section{A class of examples}\label{sec:class_of_examples}
We now turn our attention to the question of which symbols $\alpha(\xi)$ have the property that $\mathcal{F}^{-1}(1/\alpha(\xi))$ is even and strongly $PF(2)$. We first recall a sufficient condition for a function to be strictly (and therefore strongly) $PF(2)$.

\begin{lm}\label{L:PF2log} (\cite{B:albert_bona}) Suppose $K$ is positive and twice differentiable on $\mathbb R$, and $\frac{d^2}{dx^2}\left(\log(K(x))\right)<0$ for $x\neq0$. Then $K$ is strictly $PF(2)$.
\end{lm}

For a proof of this classic result, which dates to \cite{B:schoenberg}, see for example \cite{B:albert_bona}. We note that strict log-concavity is not a necessary condition for a function to be strongly $PF(2)$. For example, it is straightforward to verify that for any $a>0$, $K(x):=e^{-a|x|}$ is strongly $PF(2)$, but neither strictly log-concave nor strictly $PF(2)$. 
\vskip 10pt

As noted in \cite{B:karlin}, a general property of Polya frequency functions is that their convolutions are again Polya frequency functions.  In particular, we have: 

\begin{lm}\label{P:convolution_PF2}  Suppose $K$ and $H$ are integrable functions on $\mathbb R$, and are both $PF(2)$.  Then $K \ast H$ is $PF(2)$. If both $K$ and $H$ are strictly $PF(2)$, then $K \ast H$ is strictly $PF(2)$.  Finally, if $K$ and $H$ are continuous on $\mathbb R$ and are strongly $PF(2)$, then $K \ast H$ is strongly $PF(2)$.  
\end{lm}

\begin{proof}
Suppose $G$ and $L$ are defined on $\mathbb R^2$, and suppose the map $z \mapsto G(x,z)L(z,y)$ is integrable on $\mathbb R$ for all $(x,y) \in \mathbb R^2$.  If we define $M$ on $\mathbb R^2$ by 
\begin{equation}
M(x,y)= \int_{\mathbb R} G(x,z)L(z,y)\ dz,
\label{defM}
\end{equation}
then we have the formula
\begin{equation}
M_2(\bar x,\bar y) = \iint_\Delta G_2(\bar x, \bar z) L_2(\bar z, \bar y)\ d\bar z,
\label{bcf}
\end{equation}
which is valid for all $\bar x \in \mathbb R^2$ and $\bar y \in \mathbb R^2$.  This is a special case of a general formula which appears as formula (2.5) in chapter 1 of Karlin [1968]; the proof is elementary (expand the function on the left-hand side of \eqref{bcf} as a sum of integrals over $\Delta$ and $\mathbb R^2 \backslash \Delta$, which can be consolidated into the single integral over $\Delta$ appearing on the right-hand side).  
%Karlin in turn references problem 68 in part II of volume I of   [P\'olya and Szeg\"o, {\it Aufgaben und Lehrs\"atze aus der Analysis}, 1925].  

Now suppose $K$ and $H$ are integrable on $\mathbb R$, so that $K \ast H$ is also integrable on $\mathbb R$.  Setting $G(x,y)=K(x-y)$ and $L(x,y)=H(x-y)$, and defining $M$ by \eqref{defM}, we see via the change of variables $\tilde z = z - y$ that $M(x,y) = (K\ast H)(x-y)$.  If $K$ and $H$ are $PF(2)$, then $G$ and $L$ are $TP(2)$.  From \eqref{defM} we have $M(x,y) \ge 0$ for all $(x,y) \in \mathbb{R}^2$, and from \eqref{bcf} we have that $M_2(\bar x,\bar y) \ge 0$ for all $\bar x, \bar y \in \Delta$.  Hence $M$ is $TP(2)$, and so $K \ast H$ is $PF(2)$. The same considerations show that if $K$ and $H$ are strictly $PF(2)$, then $K \ast H$ is strictly $PF(2)$.

To prove the final assertion of the theorem, observe that for all $x \in \mathbb R$ and  all $\bar x \in \Delta$, \eqref{defM} and \eqref{bcf} imply that  
\begin{equation}
\begin{aligned}
M(x,x) &=\int_{\mathbb R} G(x,z)L(z,x)\ dz\\
M_2(\bar x,\bar x) &= \iint_\Delta G_2(\bar x, \bar z) L_2(\bar z, \bar x)\ d\bar z.
\label{bcfdiag}
\end{aligned}
\end{equation}
If $K$ and $H$ are strongly $PF(2)$, then the integrand in the first equation in \eqref{bcfdiag} is strictly positive when $z = x$, and the integrand in the second equation is strictly positive when $\bar z = \bar x$.  By continuity of $K$ and $H$, both integrands must be positive on a set of positive measure.    Since the integrands are everywhere non-negative, it follows that $M(x,x)$ and $M_2(\bar x,\bar x)$ are strictly positive.  Hence $K \ast H$ is strongly $PF(2)$.
\end{proof}

\begin{lm}\label{L:convolution_exponential_PF2} For $m\geq1$, let $K(x):=e^{-a_1|x|}*\cdots*e^{-a_m|x|}$, where $a_i>0$ for $1\leq i\leq m$. Then $K$ is even and strongly $PF(2)$.
\end{lm}

\begin{proof} As noted above, $e^{-a|x|}$ is strongly $PF(2)$ for any $a>0$, so $K$ is a convolution of even, strongly $PF(2)$ functions. The result is therefore an immediate consequence of Lemma \ref{P:convolution_PF2}.
\end{proof}

\begin{rem} For $m\geq2$, the kernel $K$ in Lemma \ref{L:convolution_exponential_PF2} is in fact log-concave. This can be verified by a direct (but lengthy) calculation in the cases $m=2$ and $m=3$, and then follows for general $m$ by induction, using the fact that convolution preserves log-concavity.
\end{rem}

\begin{cor}\label{C:L_factors_H1H2} Let $\vp$ be a solution of $\mathcal{L}\vp=f(\vp)$ where $\mathcal{L}$ has symbol $\alpha(\xi)=\prod_{i=1}^m(\xi^2+a_i^2)$ where each $a_i$ is real and nonzero. If $\vp$ and $f'(\vp)>0$ are positive and even, then $\mathcal{H}$ satisfies (H1) and (H2).
\end{cor}

\begin{proof} Since
    \[
    \mathcal{F}^{-1}(1/\alpha(\xi))=\frac{1}{2^m\prod_{i=1}^ma_i}e^{-a_1|x|}*\cdots*e^{-a_m|x|}
    \]
the result follows from Lemma \ref{L:convolution_exponential_PF2} and Theorem \ref{T:positivity}.
\end{proof}

The following lemma provides an example of a family of elliptic operators $\mathcal{L}$ with symbols $\alpha$ that satisfy \eqref{E:L} but such that $\mathcal{F}^{-1}(1/\alpha(\xi))$ is {\it not} strongly $PF(2)$.

\begin{lm}\label{L:nonPF2kernel} Let $K(x)=\displaystyle\mathcal{F}^{-1}\left(\frac{1}{\alpha(\xi)}\right)$, where $\alpha(\xi)=((\xi-\sigma)^2+\tau^2)((\xi+\sigma)^2+\tau^2)$ for some nonzero real numbers $\sigma$ and $\tau$. Then 
 \begin{equation}
 \label{nonpf2}
   K(x)=\frac1{4\sigma(\sigma^2+\tau^2)}e^{-\tau|x|}\left(\sin(\sigma x)\,\mathrm{sign}(x)+\frac{\sigma}{\tau}\cos(\sigma x)\right).
 \end{equation}
\end{lm}

\begin{proof} First note that 
	\begin{equation}
    \label{split}
		\frac{1}{((\xi-\sigma)^2+\tau^2)((\xi+\sigma)^2+\tau^2)}=\frac1{4\sigma(\sigma^2+\tau^2)}\left(\frac{\xi+2\sigma}{(\xi+\sigma)^2+\tau^2}-\frac{\xi-2\sigma}{(\xi-\sigma)^2+\tau^2}\right)
	\end{equation}
Write
	\begin{align*}
	\frac{\xi+2\sigma}{(\xi+\sigma)^2+\tau^2}&=\frac{\xi+\sigma}{(\xi+\sigma)^2+\tau^2}+\frac{\sigma}{(\xi+\sigma)^2+\tau^2}
	\end{align*}
Using the fact that $\mathcal{F}(e^{-\tau|x|})=\frac{2\tau}{\xi^2+\tau^2}$,  
	\[
	\mathcal{F}^{-1}\left(\frac{\sigma}{(\xi+\sigma)^2+\tau^2}\right)=\frac{\sigma}{2\tau}e^{-i\sigma x}e^{-\tau|x|}
	\]
and
	\[
	\mathcal{F}^{-1}\left(\frac{\xi}{\xi^2+\tau^2}\right)=-\frac{i}{2\tau}\frac{d}{dx}e^{-\tau|x|}=\frac12ie^{-\tau|x|}\,\mathrm{sign}(x)
	\]
so
	\[
	\mathcal{F}^{-1}\left(\frac{\xi+\sigma}{(\xi+\sigma)^2+\tau^2}\right)=\frac12ie^{-i\sigma x}e^{-\tau|x|}\,\mathrm{sign}(x)
	\]
and thus
	\[
	\mathcal{F}^{-1}\left(\frac{\xi+2\sigma}{(\xi+\sigma)^2+\tau^2}\right)=\frac12e^{-i\sigma x}e^{-\tau|x|}\left(i\,\mathrm{sign}(x)+\frac \sigma\tau\right).
	\]
Similarly,
	\[
	\mathcal{F}^{-1}\left(\frac{\xi-2\sigma}{(\xi-\sigma)^2+\tau^2}\right)=\frac12e^{i\sigma x}e^{-\tau|x|}\left(i\,\mathrm{sign}(x)-\frac{\sigma}{\tau}\right).
	\]
Substituting the last two equations into \eqref{split} and taking the inverse Fourier transform gives \eqref{nonpf2}.
\end{proof}

%\begin{rem} Since the kernel $K(x)$ in Lemma \ref{L:nonPF2kernel} is clearly non-positive, it is not strongly $PF(2)$. Together with Lemma \ref{L:convolution_exponential_PF2}, this implies that when $\alpha(\xi)=\xi^4+b\xi^2+c$ satisfies \eqref{E:L},  the corresponding kernel $K$ is strongly $PF(2)$ if and only if $\alpha$ factors as $(\xi^2+a_1^2)(\xi^2+a_2^2)$ for some nonzero  real numbers $a_1$ and $a_2$. An interesting question is whether this property holds, in general, for symbols $\alpha(\xi)$ satisfying \eqref{E:L} that are polynomial in $\xi^2$.
%\end{rem}

\begin{rem} Since the kernel $K(x)$ in Lemma \ref{L:nonPF2kernel} is clearly non-positive, it is not strongly $PF(2)$. Together with Lemma \ref{L:convolution_exponential_PF2}, this implies that when $\alpha(\xi)=\xi^4+b\xi^2+c$ satisfies \eqref{E:L},  the corresponding kernel $K$ is strongly $PF(2)$ if and only if $\alpha$ factors as $(\xi^2+a_1^2)(\xi^2+a_2^2)$ for some nonzero  real numbers $a_1$ and $a_2$.  More generally, when  $\alpha(\xi)$ is of the form $p(\xi^2)$, where $p$ is a polynomial, Lemma \ref{L:convolution_exponential_PF2} implies that if the roots of $p$ all lie on the negative real axis, then $K=\mathcal{F}^{-1}[1/\alpha(\xi)]$ is strongly $PF(2)$.  An interesting question is whether the converse is also true in general: if $p$ has roots which do not lie on the negative real axis, then is $K=\mathcal{F}^{-1}[1/p(\xi^2)]$ not in $PF(2)$? The classical work of Schoenberg characterizing totally positive functions in terms of the inverse Laplace transforms of their reciprocals (cf.\ Theorem 1 of \cite{B:schoenberg}) may be relevant here.
\end{rem}

\begin{comment}
% We don't make use of this fact anywhere, so I think we can leave it out.
Interestingly, however, 
	\[
	K(x)K''(x)-K'(x)^2=-\frac{e^{-2b|x|}}{16b^2(a^2+b^2)}<0
	\]
so $K$ is log-concave on intervals on which it is positive. 
\end{comment}

\vskip 10pt

 We conclude this section with the following result, which establishes the existence of positive, even solutions of \eqref{E:traveling} in the case that $\alpha$ factors as in Corollary \ref{C:L_factors_H1H2}. 

\begin{thm}\label{T:existence_positive_even} Suppose  $\mathcal{L}$ has symbol $\alpha(\xi)=\prod_{i=1}^m(\xi^2+a_i^2)$ where $a_i\neq0$ for $1\leq i\leq m$, and let $f$ satisfy Assumption \ref{A:f_assumption}. Then 
\begin{enumerate}[(a)]
\item there exists a positive solution $\vp$ of $\mathcal{L}\vp=f(\vp)$, and
\item any positive solution $\vp$ of $\mathcal{L}\vp=f(\vp)$ is a translation of some even function. 
\end{enumerate}
\end{thm}

\begin{proof} Existence of a solution of $\mathcal{L}\vp=|f(\vp)|$ in $H^m(\mathbb R)$ was shown in \cite{B:esfahani_levandosky2021}. It follows that $\vp$ is a classical solution in $C^{2m}(\mathbb R)$, and by the stable manifold theorem, $\vp$ decays exponentially to zero as $|x|\to\infty$. Thus $\vp$ satisfies $\vp=K*|f(\vp)|$ where $K(x)=\mathcal{F}^{-1}(1/\alpha(\xi))$. Since $K$ is positive, even and decreasing in $|x|$ (being a convolution of such functions), $\vp$ is positive on $\mathbb R$, and by (F1) is therefore a solution of $\mathcal{L}\vp=f(\vp)$. This establishes (a).

To prove (b), we follow the argument in Chen, Li and Ou \cite{B:chen_li_ou}. For $\lambda\in\mathbb R$ define $\Omega_\lambda=[\lambda,\infty)$, $x_\lambda=2\lambda-x$ and $u_\lambda(x)=u(x_\lambda)$. The result will follow by showing $u_\lambda=u$ on $\Omega_\lambda$ for some $\lambda\in\mathbb R$. Since $K$ is even, it follows that
\begin{equation}\label{E:uu_lambda}
\begin{split}
	u_\lambda(x)-u(x)&=\int_{\Omega_\lambda}\left(K(x-y)-K(x_\lambda-y)\right)(f(u_\lambda(y))-f(u(y)))\,dy\\
	&=\int_{\Omega_\lambda}\tilde{K}_\lambda(x,y)(f(u_\lambda(y))-f(u(y)))\,dy,
\end{split}
\end{equation}
where $\tilde{K}_\lambda(x,y)=K(x-y)-K(x_\lambda-y)$. Note that for $x,y\in\Omega_\lambda$ we have $|x-y|\leq|x_\lambda-y|$, so $K(x-y)\geq K(x_\lambda-y)>0$ and thus $\tilde{K}_\lambda(x,y)\geq0$.  Next define $\Omega_\lambda^-=\{x\in\Omega_\lambda: u(x)<u_\lambda(x)\}$ and $\Omega_\lambda^+=\{x\in\Omega_\lambda: u(x)\geq u_\lambda(x)\}$. Then for any $x\in\Omega_\lambda^-$ we have 
	\begin{align*}
	|u_\lambda(x)-u(x)|&=\int_{\Omega_\lambda^-}\tilde{K}(x,y)(f(u_\lambda(y))-f(u(y)))\,dy
	+\int_{\Omega_\lambda^+}\tilde{K}(x,y)(f(u_\lambda(y))-f(u(y)))\,dy\\
	&\leq\int_{\Omega_\lambda^-}\tilde{K}(x,y)(f(u_\lambda(y))-f(u(y)))\,dy\\
	&\leq\int_{\Omega_\lambda^-}K(x-y)(f(u_\lambda(y))-f(u(y)))\,dy\\
	&\leq C\int_{\Omega_\lambda^-}K(x-y)(|u_\lambda(y)|^{q_1-1}+|u_\lambda(y)|^{q_2-1})|u_\lambda(y)-u(y)|\,dy.
	\end{align*}
It then follows from Young's convolution inequality that
	\begin{align*}
	\|u_\lambda-u\|_{L^2(\Omega_\lambda^-)}&\leq C\|K\|_{L^2(\Omega_\lambda^-)}\|(|u_\lambda|^{q_1-1}+|u_\lambda|^{q_2-1})(u_\lambda-u)\|_{L^1(\Omega_\lambda^-)}\\
	&\leq C \||u_\lambda|^{q_1-1}+|u_\lambda|^{q_2-1}\|_{L^2(\Omega_\lambda^-)}\|u_\lambda-u\|_{L^2(\Omega_\lambda^-)}.
	\end{align*}
Since
	\[
	\||u_\lambda|^{q_1-1}+|u_\lambda|^{q_2-1}\|_{L^2(\Omega_\lambda^-)}
	\leq\left(\int_{-\infty}^\lambda(|u(y)|^{q_1-1}+|u(y)|^{q_2-1})^2\,dy\right)^{1/2}\to0
	\]
as $\lambda\to-\infty$, it follows that $\Omega_\lambda^-$ has measure zero for $\lambda<0$ sufficiently small. By \eqref{E:uu_lambda} this implies there exists $N>0$ such that 
\begin{equation}\label{E:relation1}
u_\lambda(x)\geq u(x)\text{ for all }x\geq\lambda
\end{equation} 
for all $\lambda\leq -N$. By the same reasoning, there exists $M>0$ such that 
\begin{equation}\label{E:relation2}
u_\lambda(x)\geq u(x)\text{ for all }x\geq\lambda
\end{equation} 
for all $\lambda\geq M$. Now suppose for some $\lambda_0\in[-N,M]$ we have $u(x)\geq u_{\lambda_0}(x)$ but $u(x)\not\equiv u_{\lambda_0}(x)$ for $x\geq\lambda_0$. By \eqref{E:uu_lambda} we have $u(x)> u_{\lambda_0}(x)$ for $x>\lambda_0$, so $\overline{\Omega_{\lambda_0}^-}$ has measure zero and therefore so does its reflection across $x=\lambda_0$, $\{y\leq\lambda_0:u(y)\leq u_{\lambda_0}(y)\}$. Define 
	\[
	(\Omega_\lambda^-)^*=\{x_\lambda:x\in \Omega_\lambda^-\}=\{y\leq\lambda:u_\lambda(y)<u(y)\}
	\]
to be the reflection of $\Omega_\lambda^-$ about $x=\lambda$. Then 
	\[
	 \||u_\lambda|^{q_1-1}+|u_\lambda|^{q_2-1}\|_{L^2(\Omega_\lambda^-)}=
	 \||u|^{q_1-1}+|u|^{q_2-1}\|_{L^2((\Omega_\lambda^-)^*)}.
	\]
For $\lambda>\lambda_0$, set $g_\lambda(y)=\chi_{(\Omega_\lambda^-)^*}(y)$ and suppose $y\in\{y\leq\lambda_0:u(y)\leq u_{\lambda_0}(y)\}^c$. Then either $y>\lambda_0$ in which case $y>\lambda$ for $\lambda$ sufficiently close to $\lambda_0$, or $y\leq\lambda_0$ and $u(y)> u_{\lambda_0}(y)$ in which case $u(y)- u_{\lambda}(y)>0$ for  $\lambda$ sufficiently close to $\lambda_0$. In either case we have 
$g_\lambda(y)=0$ for $\lambda$ sufficiently close to $\lambda_0$.  Hence $g_\lambda\to0$ a.e. on $\mathbb R$, so by the Dominated Convergence Theorem it follows that 
	\[
	 \||u|^{q_1-1}+|u|^{q_2-1}\|_{L^2((\Omega_\lambda^-)^*)}\to0
	\]
as $\lambda\to\lambda_0^+$. By the inequalities above it again follows that $\Omega_\lambda^-$ has measure zero for all $\lambda>\lambda_0$ sufficiently close to $\lambda_0$. That is, there exists $\epsilon>0$ such that \eqref{E:relation1} holds for $\lambda\in[\lambda_0,\lambda_0+\epsilon)$. By continuity \eqref{E:relation1} then also holds when $\lambda=\lambda+\epsilon$. Thus either $u_\lambda= u$ on $\Omega_\lambda$ for some $\lambda<M$, or \eqref{E:relation1} holds for all $\lambda\in[-N,M]$ and in particular at $\lambda=M$, which combined with \eqref{E:relation2} gives $u_M= u$ on $\Omega_M$.
\end{proof}

\vskip 10pt

\section{Two specific examples} 
\label{sec:numerical}

In this section, we apply our main results to equations of the form \eqref{E:evolution}, where $\mathcal{M}$ is a differential operator of order four or six. 

\vskip 10pt

\begin{ex} Traveling waves of the 5th order KdV equation,
	\[
	u_t-(u_{xxxx}-bu_{xx})_x+(f(u))_x=0,
	\]
satisfy
\begin{equation}\label{E:ODE4}
	\vp''''-b\vp''+c\vp=f(\vp).
\end{equation}
When $c>0$ and $b>-2\sqrt{c}$ the Fourier multiplier $\alpha(\xi)=\xi^4+b\xi^2+c$ of $\mathcal{L}$ satisfies \eqref{E:L} and therefore for $f$ satisfying Assumption \ref{A:f_assumption} there exist solutions of \eqref{E:ODE4} in $H^2(\mathbb R)$ by Theorem \ref{T:existence}. When $b\geq2\sqrt{c}$, the multiplier $\alpha$ takes the form $(\xi^2+a_1^2)(\xi^2+a_2^2)$ where $a_1=\sqrt{\frac{b+\sqrt{b^2-4c}}{2}}$ and $a_2=\sqrt{\frac{b-\sqrt{b^2-4c}}{2}}$, while for $-2\sqrt{c}<b<2\sqrt{c}$ it takes the form $((\xi-\sigma)^2+\tau^2)((\xi+\sigma)^2+\tau^2)$ where $\sigma=\sqrt{\frac{2\sqrt{c}-b}{4}}$ and $\tau=\sqrt{\frac{2\sqrt{c}+b}{4}}$. Thus, by Lemma \ref{L:convolution_exponential_PF2} the kernel $K$ defined by \eqref{E:K_definition} is strongly $PF(2)$ when $b\geq2\sqrt{c}$, and by Lemma \ref{L:nonPF2kernel}, $K$ is oscillatory when $-2\sqrt{c}<b<2\sqrt{c}$. Figure \ref{F:kernels} shows the kernel $K$ for $c=1$, and $b=-1.5,0,3$.
\begin{figure}
\begin{center}
\scalebox{0.4}{\includegraphics{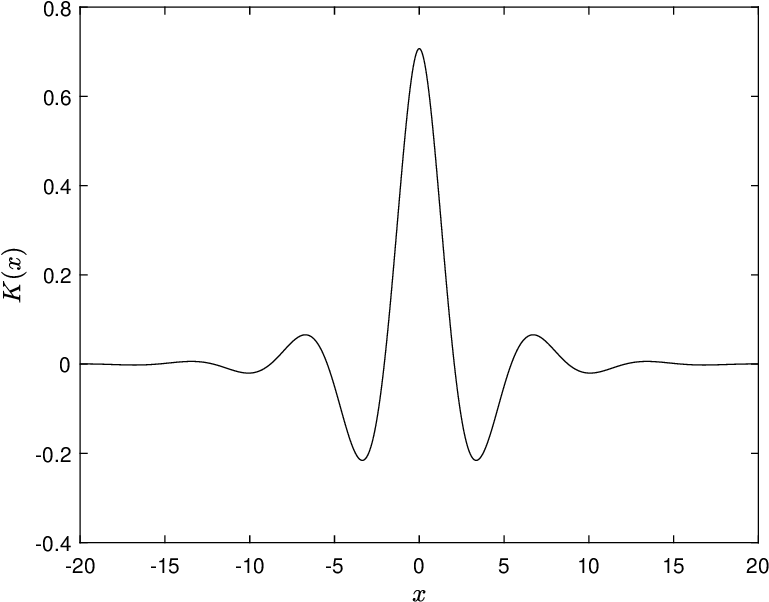}
\quad\includegraphics{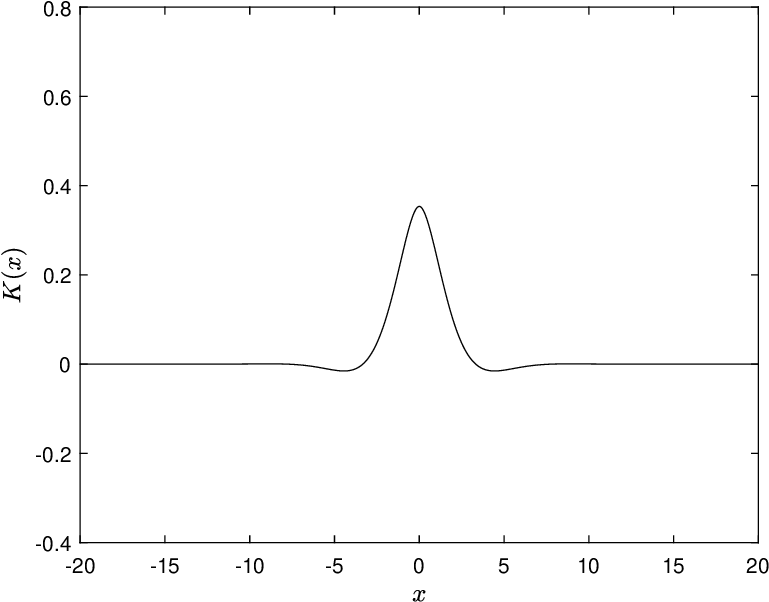}\quad\includegraphics{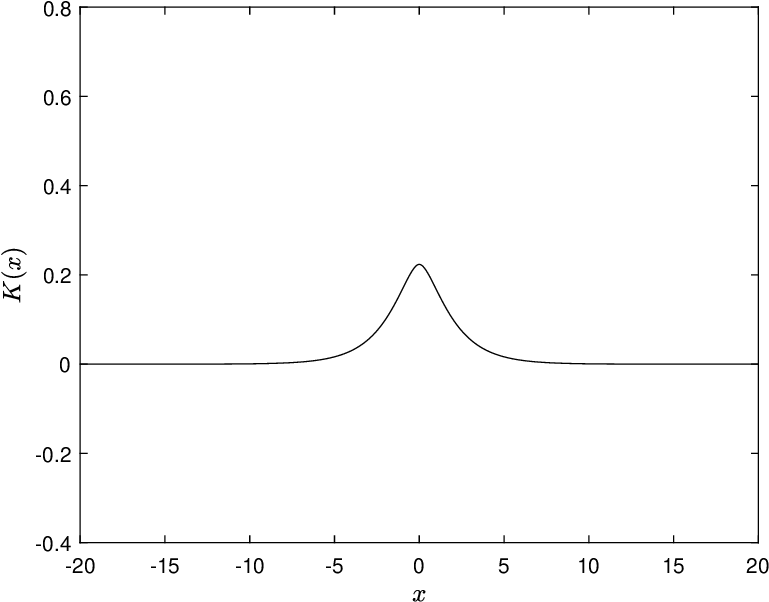}
}\caption{Kernels $K(x)=\mathcal{F}^{-1}\left(\frac1{\xi^4+b\xi^2+c}\right)$ with $c=1$, $b=-1.5,0,3$.}\label{F:kernels}
\end{center}
\end{figure}
\begin{figure}
\begin{center}
\scalebox{0.4}{\includegraphics{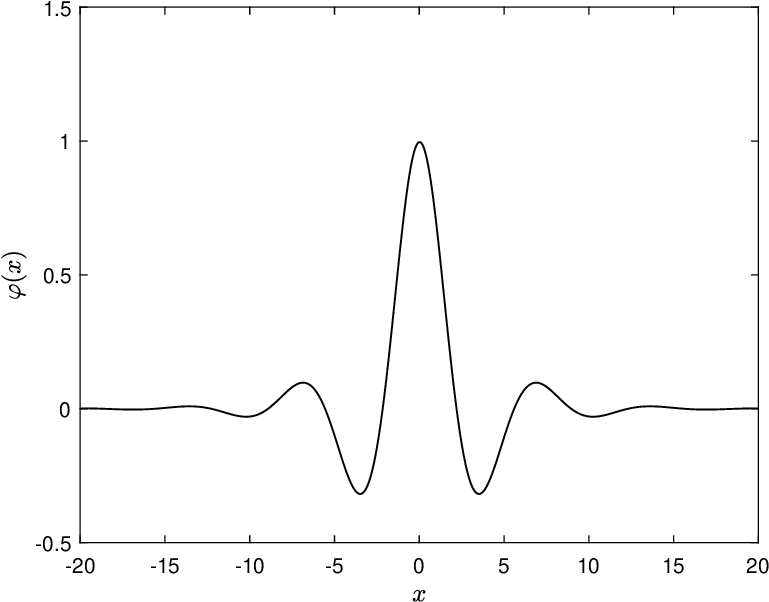}
\quad\includegraphics{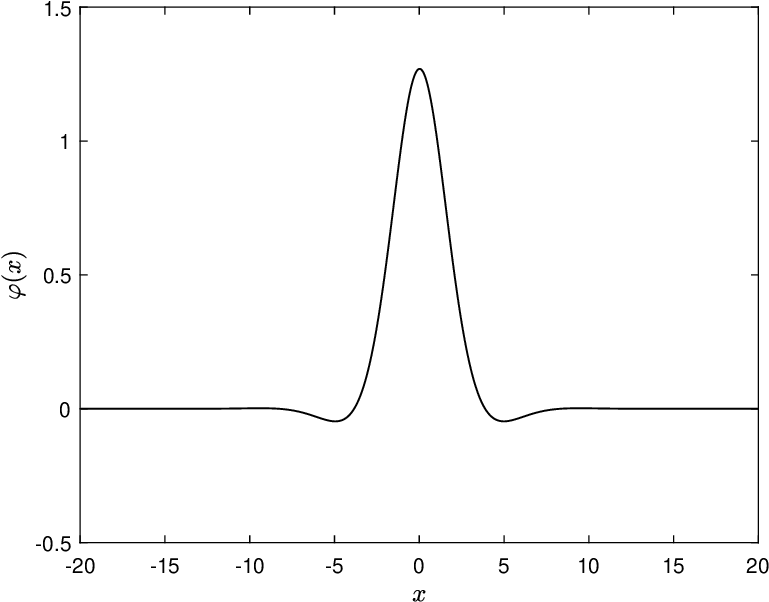}\quad\includegraphics{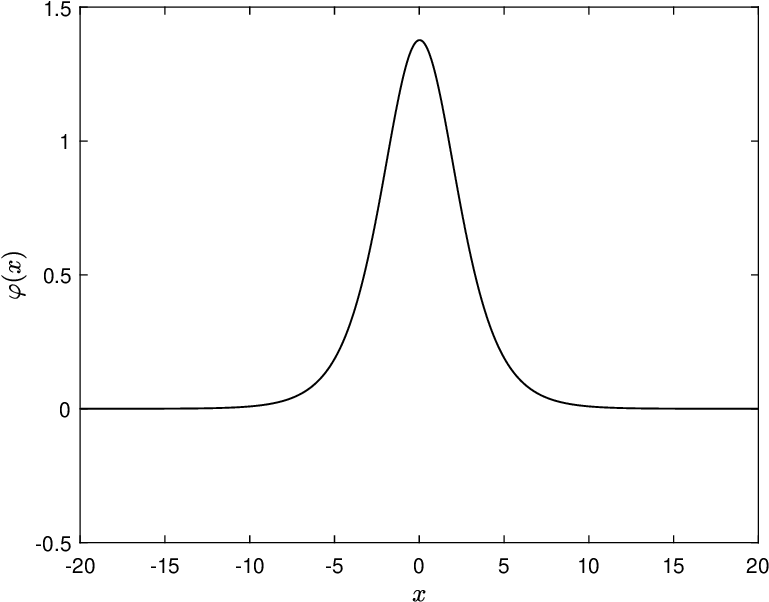}
}\caption{Solutions of $\vp''''-b\vp''+c\vp=\vp^3$ with $c=1$, $b=-1.5,0,3$.}\label{F:kdv5}
\end{center}
\end{figure}
Figure \ref{F:kdv5} shows numerical approximations of solutions of \eqref{E:ODE4} with nonlinear term $f(\vp)=\vp^3$. As one would expect, they bear a close resemblance to the corresponding kernels in Figure \ref{F:kernels}. Theorem \ref{T:existence_positive_even} implies that when $b\geq2\sqrt{c}$ there exists a positive, even solution of \eqref{E:ODE4}. For each such $\vp$, Theorem \ref{T:positivity} implies the linearized operator $\mathcal{H}$ satisfies the spectral properties (H1) and (H2). The stability or instability of $\vp$ is therefore determined by the sign of $\frac{d}{dc}\|\vp\|_{L^2}^2$. Numerical calculations were carried out in \cite{B:esfahani_levandosky2021,B:levandosky2007} to compute $\frac{d}{dc}\|\vp\|_{L^2}^2$ for various nonlinear terms satisfying (F1), (F2) and (F3). In Figure \ref{F:kdv5_domains}, the darker region contains parameter pairs $(b,c)$ for which Theorem  \ref{T:positivity} applies. To our knowledge, it is not known whether the linearized operator $\mathcal{H}$ satisfies (H1) and (H2) in the lighter shaded region.
\begin{figure}
\begin{center}
\scalebox{0.6}{\includegraphics{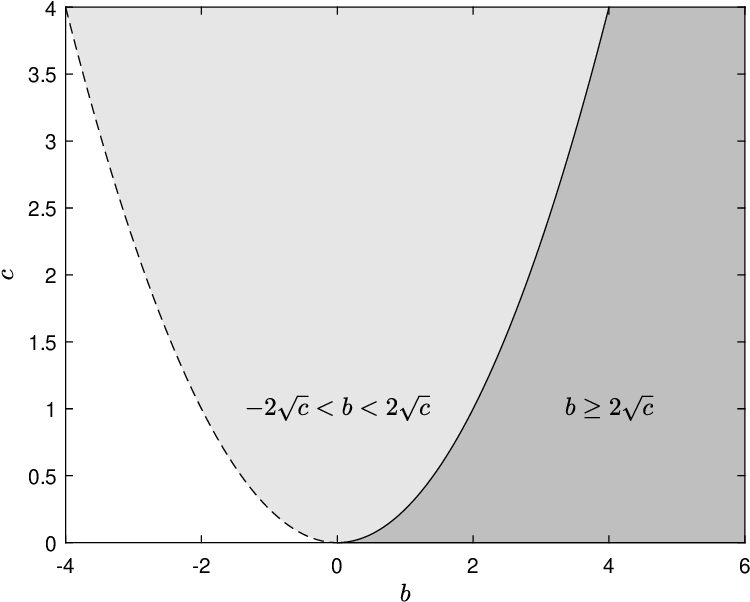}}
\end{center}
\caption{In both shaded regions, solutions of \eqref{E:ODE4} exist. The hypotheses of Theorem \ref{T:positivity} hold in the darker shaded region, where the kernel $K$ is strongly $PF(2)$.}
\label{F:kdv5_domains}
\end{figure}
\end{ex}

\begin{ex} We conclude by considering the following 7th order KdV equation \cite{B:pomeau}
	\begin{equation}\label{E:kdv7}
	u_t-(-u_{xxxxxx}+au_{xxxx}-bu_{xx})_x+(f(u))_x=0,
	\end{equation}
for which the traveling wave equation \eqref{E:traveling} becomes
	\begin{equation}\label{E:ODE6}
	-\vp''''''+a\vp''''-b\vp''+c\vp=f(\vp).
	\end{equation}
Observe that the multiplier  $\alpha(\xi)=\xi^6+a\xi^4+b\xi^2+c$ of the operator $\mathcal{L}$ factors as in Corollary \ref{C:L_factors_H1H2} if and only if it has only purely imaginary roots. Equivalently, the polynomial $p(x)=x^3+ax^2+bx+c$ must have three negative real roots, counting multiplicity, which occurs provided $c>0$, and $p'$ has real roots $r_\pm=\frac13(-a\pm\sqrt{a^2-3b})$ satisfying $r_-\leq r_+<0$, and $p(r_+)\leq0\leq p(r_-)$. We have $r_+<0$ when $a>0$ and $0< b\leq\frac13a^2$, and we have $p(r_+)\leq0\leq p(r_-)$ when $c>0$ and $h_-(a,b)\leq c\leq h_+(a,b)$, where
	\[
	h_\pm(a,b):=\frac2{27}(a^2-3b)(-a\pm\sqrt{a^2-3b})+\frac19ab.
	\]
Thus the region of parameters for which $\mathcal{L}$ factors as a composition of second order elliptic operators takes the form $F=\{(a,b,c):a>0, b>0, \max\{0,h_-(a,b)\}\leq c\leq h_+(a,b)\}$. Next, we note that $\alpha$ satisfies the ellipticity condition \eqref{E:L} with $s=3$ if and only if $p$ is positive for $x\geq0$. This clearly requires $c>0$. Note that since $p(x)=x(x^2+ax+b)+c$, $p$ is positive for $x\geq0$ if $c>0$, $b\geq0$ and $a\geq-2\sqrt{b}$. On the other hand, if either $b<0$, or $b\geq0$ and $a\leq-2\sqrt{b}$, then $p'$ has real roots $r_\pm$ where $r_+>0$, and therefore $p$ is positive for $x\geq0$ if and only if $p(r_+)>0$, which holds when $c>h_+(a,b)$. Thus $\mathcal{L}$ is elliptic when $(a,b,c)$ lies in the set $E$ of points that satisfy $c>h(a,b)$ where
	\[
	h(a,b)=\left\{
	\begin{array}{cc}
	0 & b\geq0, a\geq-2\sqrt{b}\\
	h_+(a,b) & \text{otherwise}
	\end{array}
	\right.
	\]
By a direct calculation one can show that $F=\{(a,b,c):a>0,b>0,c>0, D(a,b,c)\geq0\}$, where 
	\[
	D(a,b,c):=a^2b^2-4b^3-4a^3c-27c^2+18abc
	\]
is the discriminant of the polynomial $p$. Similarly, $E=\{(a,b,c):D(a,b,c)<0\}\cup\{(a,b,c):a>0,b>0,c>0\}$. Figure \ref{F:domains} shows the surfaces $c=h_\pm(a,b)$ that bound the region $F$. Figure \ref{F:kdv7regions} shows the cross-sections of $E$ and $F$ with $c=1$. The darker region is the cross-section with $F$ and the two regions combined is the cross-section with $E$. Since the sign of the discriminant is invariant under the scaling $(a,b,c)\to(sa,s^2b,s^3c)$, cross sections with other $c>0$ are obtained by applying the mapping $G(a,b)=(c^{1/3}a,c^{2/3}b)$ to the regions in Figure \ref{F:kdv7regions}.

\begin{figure}
\begin{center}
	\scalebox{0.6}{\includegraphics{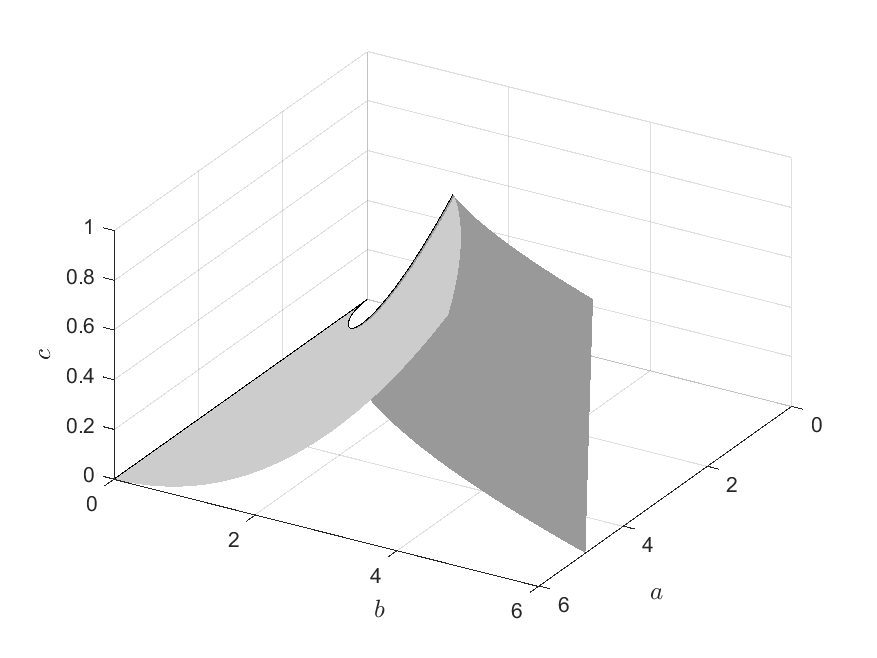}}\qquad
\end{center}
\caption{The symbol $\alpha(\xi)=\xi^6+a\xi^4+b\xi^2+c$ of the operator factors into $(\xi^2+a_1^2)(\xi^2+a_2^2)(\xi^2+a_3^2)$ for real nonzero $a_i$ in the region between the two surfaces.}\label{F:domains}
\end{figure}

\begin{figure}
\begin{center}
	\scalebox{0.6}{\includegraphics{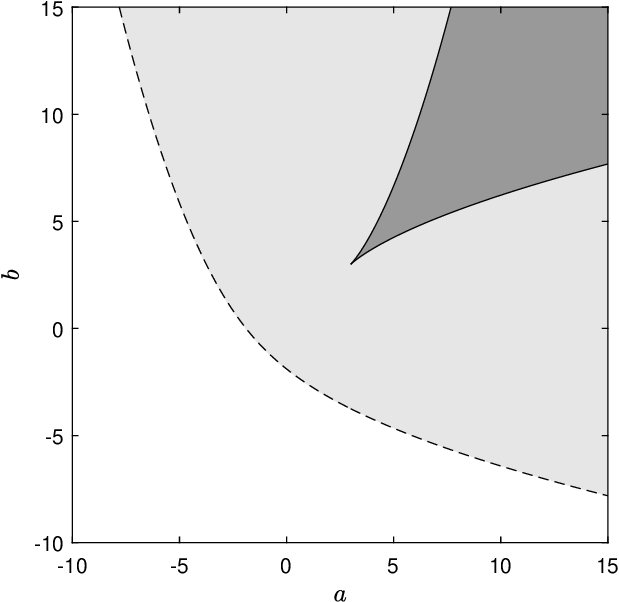}}
\end{center}
\caption{When $c=1$, the symbol $\alpha(\xi)=\xi^6+a\xi^4+b\xi^2+c$ of the operator $\mathcal{L}$ satisfies \eqref{E:L} for $(a,b)$ in the two shaded regions, and factors into $(\xi^2+a_1^2)(\xi^2+a_2^2)(\xi^2+a_3^2)$ for real nonzero $a_i$ in the darker shaded region.}\label{F:kdv7regions}
\end{figure}

As a concrete example, consider $a=6$ and $b=9$. Since $h_-(6,0)=0$ and $h_+(6,9)=4$, the operator $\mathcal{L}$ factors as in Corollary \ref{C:L_factors_H1H2} if $0<c<4$. For $c$ in this range, \eqref{E:ODE6} has positive solutions and the spectral conditions (H1) and (H2) hold.  One may then use the sign of $\frac{d}{dc}\|\vp\|_{L^2}^2$ to determine whether these solutions are stable under the evolution of \eqref{E:kdv7}, with a positive sign implying stability and a negative sign implying instability.  Figure \ref{F:kdv7_stability} shows numerical approximations of $\frac{d}{dc}\|\vp\|_{L^2}^2$ in the case of a pure power nonlinearity $f(\vp)=|\vp|^{p-1}\vp$ with $p=2,4,6,8$. Taking these into account, we conclude that when $p=2$ and $p=4$ traveling waves with speed $c\in(0,4)$ are stable, when $p=8$ they are unstable, and when $p=6$ there is a critical speed $c^*\approx1.2$ such that they are unstable for $c\in(0,c^*)$ and stable for $c\in(c^*,4)$. 

\end{ex}

\begin{figure} [H]
\begin{center}
\scalebox{0.5}{\includegraphics{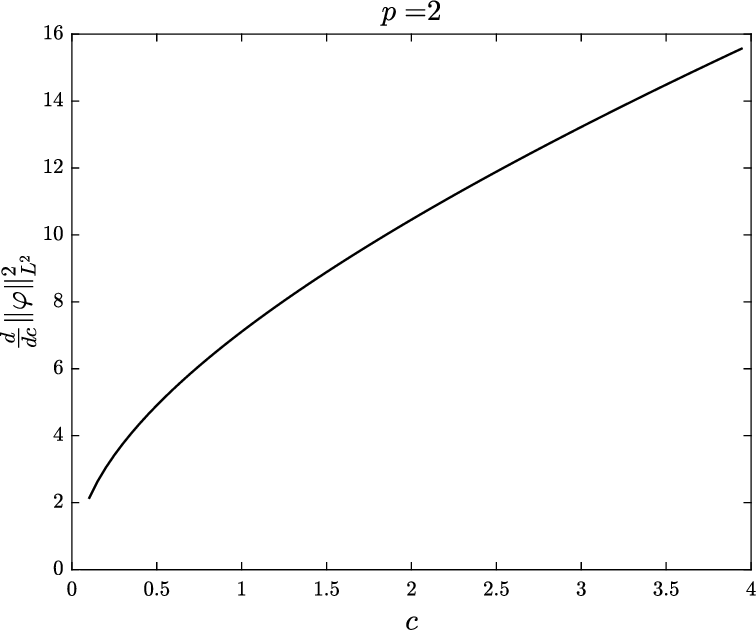}
\quad
\includegraphics{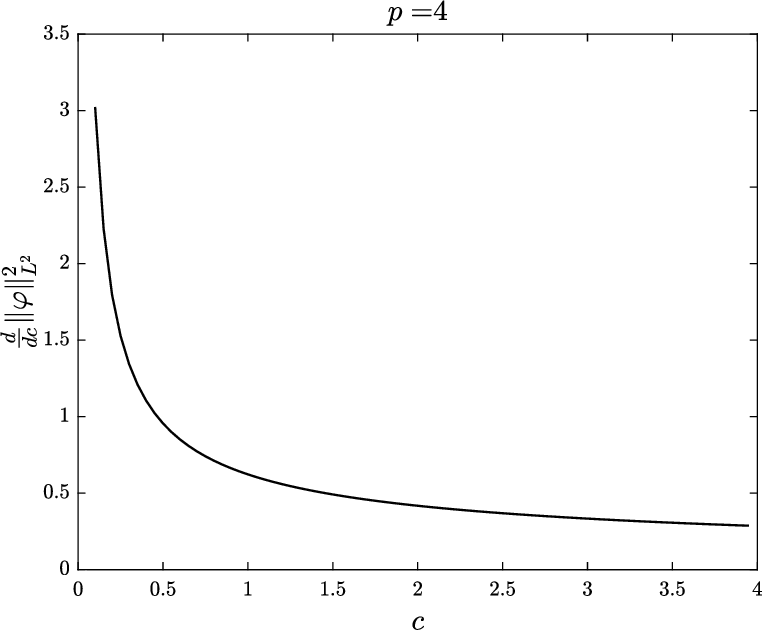}}\\
\scalebox{0.5}{\includegraphics{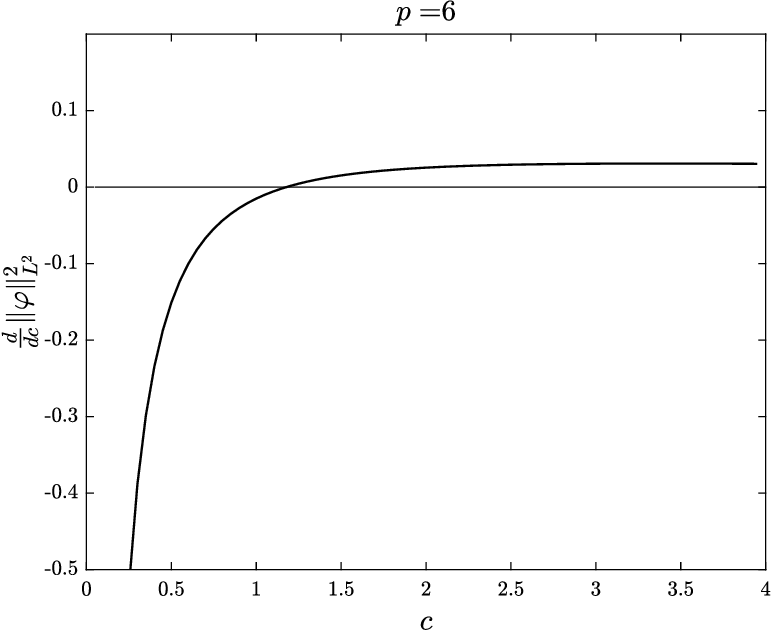}
\quad
\includegraphics{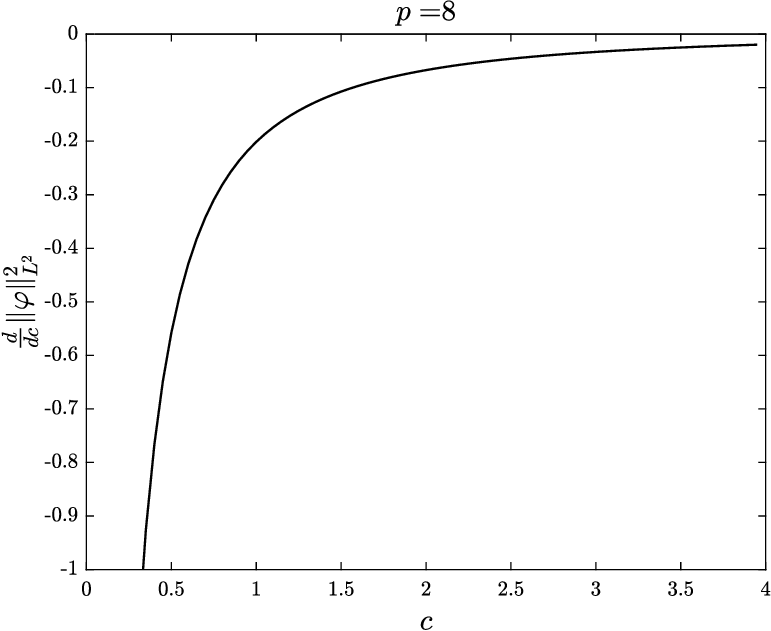}}
\end{center}
\caption{Numerical approximations of $\frac{d}{dc}\|\varphi\|_{L^2}^2$ for solutions of \eqref{E:ODE6} with $a=6$, $b=9$ and $f(\vp)=|\vp|^{p-1}\vp$ with $p=2,4,6,8$.}\label{F:kdv7_stability}
\end{figure}

\bigskip
  \footnotesize

  J.~Albert, \textsc{Department of Mathematics, University of Oklahoma,
    Norman OK 73019}\par\nopagebreak
  \textit{E-mail address;} \texttt{jalbert@ou.edu}

  \medskip

  S.~Levandosky, \textsc{Mathematics and Computer Science Department, College of the Holy Cross, Worcester, MA 01610}\par\nopagebreak
  \textit{E-mail address:} \texttt{slevando@holycross.edu}
  
  \end{document}